\documentclass[preprint,12pt]{elsarticle}

\usepackage{a4wide,amsfonts, amsmath, amscd}
\usepackage[psamsfonts]{amssymb}

\usepackage{graphicx}
\usepackage[cp1251]{inputenc}
\usepackage{amssymb}
\usepackage{amsthm}
\usepackage{cmap}

\usepackage{amsmath}

\biboptions{sort&compress}

\newcommand{\sysn}{\left\{\begin{array}{rcl}}
\newcommand{\sysk}{\end{array}\right.}

\newtheorem{theorem}{Theorem}[section]
\newtheorem{lemma}[theorem]{Lemma}

\theoremstyle{example}

\newtheorem{proposition}[theorem]{Proposition}
\theoremstyle{definition}
\newtheorem{definition}[theorem]{Definition}

\newtheorem{corollary}[theorem]{Corollary}

\journal{...}

\begin{document}

\title{On the product of almost discrete Grothendieck spaces}


\author{Alexander V. Osipov}


\ead{OAB@list.ru}

\address{Krasovskii Institute of Mathematics and Mechanics, Yekaterinburg, Russia}


\address{Ural Federal University, Yekaterinburg, Russia}

\begin{abstract} A topological space $X$ is called almost discrete, if it has precisely one nonisolated point. In this paper, we get that for a countable product $X=\prod X_i$ of almost discrete spaces $X_i$  the space $C_p(X)$ of continuous real-valued functions with the topology of pointwise convergence is a $\mu$-space if, and only if, $X$ is a weak $q$-space if, and only if, $t(X)=\omega$ if, and only if, $X$ is functionally generated by the family of all its countable
subspaces.

This result makes it possible to solve Archangel'skii's problem on the product of Grothendieck spaces. It is proved that in the model of $ZFC$, obtained by adding one
Cohen real, there are Grothendieck spaces $X$ and $Y$ such that $X\times Y$ is not weakly Grothendieck space. In $(PFA)$: the product of any countable family almost discrete Grothendieck spaces is a Grothendieck space.

\end{abstract}



\maketitle 


\section{Introduction}

In 1952 Grothendieck \cite{Grot} proved the following result.

\begin{theorem}(Grothendieck) Let $X$ be a compact space and let $Y$ be a metrizable
space. Then each relatively countably compact subspace
 of $C_p(X,Y)$ is relatively compact.
\end{theorem}

This theorem has played an important role in topology, mathematical and
functional analysis. Grothendieck's theorem has been generalized many times \cite{Os,Arch1,Chr,PS,Pt}.

\medskip
A topological space $X$ is called {\it almost discrete}, if it has precisely one nonisolated point.

\medskip
A space $X$ is called {\it $g$-space}, if for every subset $A$ of $X$ such that $A$ is countably compact in $X$, the closure of $A$ in $X$ is compact. A space $X$ is called {\it a Grothendieck space (a weakly Grothendieck space)}, if $C_p(X)$ is a hereditary $g$-space (a $g$-space).

The product of two weakly Grothendieck spaces need not be a weakly Grothendieck space (see Proposition \ref{pr11}).

\medskip

In 1998, A.V. Archangel'skii posed the following question (Question 5.13 in \cite{Arh98}, see also Problem 4.8.3 in \cite{Tk}):

{\it Is the product of two Grothendieck spaces a weakly Grothendieck? A Grothendieck space?}

In this paper we get the following results:

\medskip

$\bullet$ Countable product $X$ of almost discrete spaces is weakly Grothendieck if, and only if, the tightness of $X$ is countable (Th. \ref{th1}).

\medskip

$\bullet$ Countable product $X$ of almost discrete spaces is Grothendieck if, and only if,  the tightness of $X$ is countable and $X$ is Lindel\"{o}f (Th. \ref{th61}).

\medskip

$\bullet$ In the model of $ZFC$, obtained by adding one Cohen real, there are almost
discrete Grothendieck spaces $Y_0$ and $Y_1$ such that $Y_0\times Y_1$ is not weakly Grothendieck (Th. \ref{th65}).

\medskip

$\bullet$ $(PFA)$ Countable product of almost discrete Grothendieck spaces is a Grothendieck space. (Th. \ref{th68}).

\section{Notation and terminology}
 The set of positive integers is denoted by $\mathbb{N}$ and
$\omega=\mathbb{N}\cup \{0\}$. Let $\mathbb{R}$ be the real line. We denote by $\overline{A}$ the closure of $A$ (in $X$).

Recall that a subset $A$ of a topological space $X$ is called

$\bullet$ {\it relatively compact} if $A$ has a compact closure in $X$;

$\bullet$ {\it bounded} if every continuous function on $X$ is bounded on $A$;

$\bullet$  {\it countably compact} if each sequence of $A$ has a
limit point in $A$;

$\bullet$  {\it relatively countably compact} if each sequence of $A$ has a
limit point in $X$;

$\bullet$  {\it pseudocompact} if any continuous real-valued function on $A$ is bounded;

$\bullet$  {\it countably pracompact} subspace of $X$, if there is a subspace $Y\subseteq A$ which is dense in $A$ and countably compact in $A$ in the following sense: every infinite set $B\subseteq Y$ has a limit point in $A$.
\medskip

A topological space $X$ is called {\it $\mu$-space}, if each bounded subset of $X$ is relatively compact.

\medskip
Let $X$ be a Tychonoff topological space, $C(X,\mathbb{R})$ be the
space of all  continuous functions on $X$ with values in
$\mathbb{R}$ and $\tau_p$ be the pointwise convergence topology.
Denote by $C_p(X)$ the topological space
$(C(X,\mathbb{R}), \tau_p)$.

\medskip

It will be convenient for us to use (for brevity) the notation proposed by E. Reznichenko for the following general property.

\medskip

Let $\mathcal{P}$ be a family subsets of $C_p(X)$. If $A$ is relatively compact in $C_p(X)$ for any $A\in \mathcal{P}$, then we will write that $X$ is a {\it  $\mu^{\sharp}_{\mathcal{P}}$-space}.

\medskip

In this paper we consider the following families $\mathcal{P}$ of subsets of $C_p(X)$:

 ${\bf rc}$ --- family of all relatively compact subsets;

 ${\bf b}$ --- family of all bounded subsets;

  ${\bf cc}$ --- family of all countably compact subsets;

 ${\bf rcc}$ --- family of all relatively countably compact subsets;

  ${\bf pcc}$ --- family of all pseudocompact subsets;

 ${\bf cp}$ --- family of all countably pracompact subsets.

\medskip

Note that $C_p(X)$ is a $\mu$-space equivalent to $X$ is a $\mu^{\sharp}_{b}$-space.

Also note that weakly Grothendieck spaces in this terminology are $\mu^{\sharp}_{cc}$-spaces.

\medskip

Let $\mathcal{M}$ be a family of subspaces of a space $X$. We say that
$X$ is {\it functionally generated by $\mathcal{M}$} if for any discontinuous function $f: X \rightarrow \mathbb{R}$
there is an $A\in\mathcal{M}$ such that $f|A$ cannot be extended to a continuous real-valued function on the whole of $X$ (Definition 2.11 in \cite{Arh33}).

For example, $X$ is functionally generated by the family of all its countable
subspaces if and only if its weak functional tightness $t_{\bf R}$ is countable: $t_{\bf R}(X)=\omega$; this
follows directly from the definition of  weak functional tightness $t_{\bf R}$  in \cite{arch}.

The {\it tightness} $t(X)$ of $X$ is the smallest infinite cardinal $\tau$ such that for any set $A\subset X$ and any point $x\in X$ there is a set $B\subset X$ for which $|B|\leq \tau$ and $x\in \overline{B}$.

The notation of a weak $q$-space defined in \cite{AsVel}. Recall that a space is called {\it realcomplete} if it is homeomorphic to a closed subspace of the space $\mathbb{R}^{\tau}$ for a certain $\tau$.

For other notation and terminology almost without exceptions we follow the Engelking's book \cite{Eng}.

\section{Countable product of almost discrete spaces}

\begin{lemma} (Corollary II.4.17 in \cite{arch})\label{lem0} If $t(X)=\omega$, then $C_p(X)$ is realcomplete and, hence, $C_p(X)$ is a $\mu$-space.

\end{lemma}

Let $f:X\rightarrow Y$ be a map (between {\it sets} $X$ and $Y$). Define the map $f^{\sharp}:\mathbb{R}^Y\rightarrow \mathbb{R}^X$ (between {\it topological spaces}) dual to $f$ as follows: if $\phi\in \mathbb{R}^Y$, then $f^{\sharp}(\phi)(x)=\phi(f(x))$ for all $x\in X$, i.e. $f^{\sharp}(\phi)=\phi\circ f$.

\begin{lemma}\label{lem2} Let $X$ and $Y$ be topological spaces, $f:X\rightarrow Y$ be a quotient map from $X$ onto $Y$. Then, $Y$ is a weakly Grothendieck space provided $X$ is a weakly Grothendieck space.
\end{lemma}

\begin{proof} By Proposition 0.4.8 (2) in \cite{arch}, $f^{\sharp}(C_p(Y))$ is a closed subset of $C_p(X)$. Hence,  $f^{\sharp}(\overline{A}^{C_p(Y)})=\overline{f^{\sharp}(A)}^{C_p(X)}$ for any $A\subset C_p(Y)$.
Note that $\overline{f^{\sharp}(A)}^{C_p(X)}$ is compact for any countably compact subset $A$ of $C_p(Y)$.

Since $f$ is surjection, by Proposition 0.4.6 in \cite{arch}, $f^{\sharp}$ is a homeomorphism. Hence, $\overline{A}^{C_p(Y)}$ is compact.

\end{proof}

\newpage

\begin{theorem}\label{th1} Let $X$ be a countable product of almost discrete spaces. Then the following statements are equivalent:

\begin{enumerate}

\item $C_p(X)$ is a $\mu$-space;

\item $C_p(X)$ is realcomplete;

\item  $X$ is a $\mu^{\sharp}_{pc}$-space;

\item  $X$ is a $\mu^{\sharp}_{rcc}$-space;

\item  $X$ is a $\mu^{\sharp}_{cp}$-space;

\item  $X$ is a weakly Grothendieck space;

\item $t(X)=\omega$;

\item $t_{\bf R}(X)=\omega$;

\item $X$ is a weak $q$-space.

\end{enumerate}

\end{theorem}

\begin{proof} The following implications are trivial: $(1)\Rightarrow(3)\Rightarrow(5)\Rightarrow(6)$; $(1)\Rightarrow(4)\Rightarrow(5)$; $(7)\Rightarrow(8)$; $(7)\Rightarrow(9)$.

By Theorem 2.16 in \cite{Arh33}, $(8)\Rightarrow(1)$. By Theorem 1 in \cite{AsVel}, $(9)\Rightarrow(1)$. By Corollary II.4.17 in \cite{arch}, $(2)\Leftrightarrow (8)$.

$(6)\Rightarrow(7)$. (a) Let $X=\prod\limits_{i=1}^n X_i$ be a finite product of almost discrete spaces $X_i$ $(i=1,...,n)$ and $X$ is a weakly Grothendieck space.
Claim that $t(X)=\omega$.

$\bullet$ Let $n=1$. Assume that $t(X)> \omega$.

Let $x_0$ be a nonisolated point of $X$. Then there is $A\subset X$ such that $|A|>\omega$, $x_0\in \overline{A}\setminus A$ and any countable subset of $A$ is closed in $X$.

Let $F=\{f\in C_p(X): f(X\setminus S)=0$, $f(S)=1$, $S\in [A]^{\omega}\}$. Note that any $S\in [A]^{\omega}\}$ is clopen subset of $X$.

Claim that $F$ is countably compact and, hence, is bounded in $C_p(X)$. Indeed, if $\{f_n: n\in \mathbb{N}\}\subseteq F$, and $B=\bigcup \{f^{-1}_n(\{1\}): n\in \mathbb{N}\}$ then the function

$$ f_0(x)=\left\{
\begin{array}{lcl}
0 \, \, \, for \, \, \, x\notin B;\\
s(x) \, \, \, for \, \, \, x\in B\\
\end{array}
\right.
$$

where $s(x)$ is a limit point the sequence $\{f_n|B : n\in \mathbb{N}\}\subset 2^B$ belongs to $F$ and is an accumulation point for the set $\{f_n: n\in \mathbb{N}\}$.

On the other hand, $\overline{F}^{C_p(X)}$ is not compact because the function

$$ h(x)=\left\{
\begin{array}{lcl}
0 \, \, \, for \, \, \, x\notin A;\\
1 \, \, \, for \, \, \, x\in A\\
\end{array}
\right.
$$

is discontinuous and $h\in \overline{F}^{\mathbb{R}^X}$.

\medskip

$\bullet$ Let $n>1$ and the theorem is true for any $k<n$.

Let $X=X_1\times...\times X_{n}$ and $X$ is a weakly Grothendieck space. Denote by $x_i$ is one nonisolated point in $X_i$ for every $i=1,...,n$. Consider the subspace $S=\bigcup \{(X_1\times...\times X_{i-1}\times \{x_i\}\times X_{i+1}\times...\times X_n): i=1,...,n\}$. The set $S$ is closed in $X$ and any point in $X\setminus S$ is a isolated point in $X$.
Note that $S$ is a quotient image of the space $Q=\oplus \{(X_1\times...\times X_{i-1}\times X_{i+1}\times...\times X_n): i=1,...,n\}$. By induction,  $t(Q)=\omega$ and $t(S)=\omega$ (Proposition 3 in \cite{ArchPon}).

Let $Z$ be a quotient space of $X$ with $S$ identified to a point. Let $\pi: X\rightarrow Z$ be the quotient map.
Note that $\pi$ is a closed map and $\{S\}$ is one nonisolated point in $Z$. By Lemma \ref{lem2}, $Z$ is a weakly Grothendieck space and, hence, $t(Z)=\omega$.
Since $\pi$ is a continuous closed surjective map, $t(X)=\sup \{t(Z), t(\pi^{-1}(x)): x\in Z\}$ (Theorem 4.5 in \cite{Arch79}). It follows that $t(X)=\omega$.

\medskip

(b) Let $X=\prod X_i$ be a countable product of almost discrete spaces $X_i$ and $X$ is a weakly Grothendieck space.
Let $\pi_{i_1,...,i_k}: X\rightarrow X_{i_1}\times...\times X_{i_k}$ be a projection function from $X$ onto $X_{i_1}\times...\times X_{i_k}$. Then $\pi_{i_1,...,i_k}$ is a quotient map and,
by Lemma \ref{lem2}, $X_{i_1}\times...\times X_{i_k}$ is a weakly Grothendieck space. By (a), $t(X_{i_1}\times...\times X_{i_k})=\omega$. By Remark 3 in \cite{Mal2}, $t(X)=\omega$.

\end{proof}

\begin{corollary}\label{cor2} An almost discrete space $X$ is weakly Grothendieck if and only if $t(X)=\omega$.
\end{corollary}

\begin{theorem}(Theorem 5.6 in \cite{Arh98})\label{th33} A weakly Grotendieck space $X$ is a Grothendieck space if and only if every compact subspace of $C_p(X)$ is Fr\'{e}chet-Urysohn.
\end{theorem}


\medskip
Let $\tau$ be an infinite cardinal, $D(\tau)$ the discrete space of cardinality $\tau$, $\xi\notin D(\tau)$, and $L(\tau)=D(\tau)\cup \{\xi\}$ a space in which only the point $\xi$ is not isolated, and where the neighborhoods of $\xi$ are all sets $V\subset L(\tau)$ such that $\xi\in V$ and $L(\tau)\setminus V$ is countable. Clearly, $L(\tau)$ is a Lindel\"{o}f $P$-space.

Recall that a space $Y$ is called {\it a primary Lindel\"{o}f space}, if $Y$ is a continuous image of closed subspace of the space $(L(\tau))^{\omega}$.


\begin{theorem}\label{th61} A countable product $X=\prod X_i$ of almost discrete spaces $X_i$ is Grothendieck if, and only if, $t(X)=\omega$ and $X$ is Lindel\"{o}f.
\end{theorem}

\begin{proof} $(\Rightarrow)$. By Theorem \ref{th1}, $t(X)=\omega$. Every continuous image of a Grothendieck space is Grothendieck (Theorem 5.4 in \cite{Arh98}). Then, $X_i$ is Grothendieck.
It remains to note that if an almost discrete space is Grothendieck, then it is a continuous bijective image of one-point
Lindel\"{o}fication $L(\tau)$ of a discrete space. Note that the space $(L(\tau))^{\omega}$ is Lindel\"{o}f (Proposition IV.2.19 in \cite{arch}). It follows that $X$ is Lindel\"{o}f.

$(\Leftarrow)$. By Theorem \ref{th1}, $X$ is a weakly Grothendieck space. Note that an almost discrete Lindel\"{o}f space is a continuous bijective image of one-point
Lindel\"{o}fication of discrete space. Thus, for every $i\in \mathbb{N}$ there is a one-point
Lindel\"{o}fication $Y_i$ of discrete space of cardinality $|X_i|$  such that $X_i$ is a bijective continuous image of $Y_i$.

The product $Y=\prod\limits Y_i$ is a primary Lindel\"{o}f space.  Then, $C_p(Y)$ can be mapped by a one-to-one continuous linear map into the $\Sigma$-product of a certain number of copies of
the real line (Proposition IV.3.10 in \cite{arch}). Since $C_p(X)$ is homeomorphic to a subset of $C_p(Y)$, every compact subset of $C_p(X)$ is a Corson compactum and, hence, it is Fr\'{e}chet-Urysohn. By Theorem \ref{th33}, $X$ is Grothendieck.
\end{proof}

\begin{corollary}\label{cor1} An almost discrete space $X$ is Grothendieck if, and only if,  $X$ is Lindel\"{o}f and $t(X)=\omega$.
\end{corollary}

\section{Examples}

In \cite{Arh98}, A.V. Archangel'skii noted the following result.

\begin{proposition}\label{pr11} There are almost discrete Fr\'{e}chet-Urysohn spaces $X$ and $Y$ such that $X\times Y$ is not weakly Grothendieck space.
\end{proposition}

\begin{proof}  Consider $X$ is the countable sequential fan $S_{\omega}$ and $Y$ is the sequential fan $S_{\mathfrak{c}}$ of cardinality $2^{\omega}$. Then, $t(X\times Y)>\omega$. By Theorem \ref{th1}, $X\times Y$ is not  weakly Grothendieck space.
\end{proof}

Note that Archangel'skii's example such that each of the factors is a closed image of the metric space. We can provide an alternative proof of Proposition \ref{pr11}.
Consider the spaces $X$ and $Y$ in Example 1 in \cite{GN}. Then, $X$ and $Y$ are almost discrete Fr\'{e}chet-Urysohn spaces and $t(X\times Y)>\omega$.
By Theorem \ref{th1}, $X\times Y$ is not weakly Grothendieck space. Note that in our example the closure of a countable set is a subspace with the first axiom of countability.





\begin{proposition} There is an almost discrete Fr\'{e}chet-Urysohn space $Z$ such that $Z\times Z$ is not weakly Grothendieck space.
\end{proposition}

\begin{proof} Let $X$ and $Y$ be spaces from Proposition \ref{pr11} where $x$ and $y$ are non-isolated points of $X$ and $Y$, respectively.
Let $Z$ be a quotient space of $X\bigcup Y$ with $\{x,y\}$ identified to a point.
Then $Z$ is an almost discrete Fr\'{e}chet-Urysohn space. Since $Z\times Z$ contains $X\times Y$, $t(Z\times Z)>\omega$. By Theorem \ref{th1}, $Z\times Z$ is not weakly Grothendieck space.
\end{proof}

\section{Archangel'skii's problem on the product of Grothendieck spaces}



\begin{definition}(\cite{Mal}) A pair $(Y_0,Y_1)$ of spaces $Y_0$ and $Y_1$ is called {\it a Leiderman pair}, if

(1) $Y_0$ and $Y_1$ are almost discrete spaces of cardinality $\omega_1$;

(2) $Y_0$ and $Y_1$ are Lindel\"{o}f;

(3) $t(Y_0)=t(Y_0^{\omega})=t(Y_1)=t(Y_1^{\omega})=\omega$;

(4) $t(Y_0\times Y_1)>\omega$.

\end{definition}

The following theorem answers the question posed.

\begin{theorem}\label{th65} In the model of $ZFC$, obtained by adding one
Cohen real, there are almost discrete Grothendieck spaces $Y_0$ and $Y_1$ such that $Y_0\times Y_1$ is not weakly Grothendieck space.
\end{theorem}

\begin{proof} By Theorem 2 in \cite{LM}, there is a Leiderman pair $(Y_0,Y_1)$ of almost discrete spaces $Y_0$ and $Y_1$.
By Corollary \ref{cor1}, $Y_0$ and $Y_1$ are Grothendieck spaces.
Since $t(Y_0\times Y_1)>\omega$, by Theorem \ref{th1}, $Y_0\times Y_1$ is not weakly Grothendieck space.
\end{proof}

Similarly  Theorem 2 in \cite{LM}, we note the following result.

\begin{proposition}
The assertion of the existence of two almost discrete Grothendieck spaces (a Leiderman pair) for which their product is not weakly Grothendieck space, is consistent with any cardinal arithmetic (including with $CH$ and also those with the negation of $CH$).
\end{proposition}

We shall say that a (completely regular) space $X$ is {\it condensed}, if there is a single point $x\in X$ such that every uncountable subset of $X$ accumulates to $x$.

\begin{theorem} (Theorem 9 in \cite{Tod})\label{Tod} $(PFA)$ If $X$ is a countably tight condensed space, then its countable power $X^{\omega}$ is also countably tight.
\end{theorem}

\begin{theorem}\label{th68} $(PFA)$ The product of any countable family almost discrete Grothendieck spaces is a Grothendieck space.
\end{theorem}

\begin{proof} Let $Y=\prod Y_{i}$ be a countable product of  almost discrete Grothendieck spaces $Y_{i}$.
By Corollary \ref{cor1}, $Y_{i}$ is Lindel\"{o}f and $t(Y_{i})=\omega$ for every $i\in \mathbb{N}$.
By Theorem \ref{th61}, we just need to check that $t(Y)=\omega$.
Let $y_i$ be one non-isolated point of $Y_i$ for every $i\in \mathbb{N}$. Let $Z$ be a quotient space of $\bigcup \{Y_i: i\in \mathbb{N}\}$ with $\{y_i: i\in \mathbb{N}\}$ identified to a point.
Then $Z$ is an almost discrete Lindel\"{o}f space and $t(Z)=\omega$. By Theorem \ref{Tod}, $t(Z^{\omega})=\omega$. Since $Y$ is homeomorphic to a subset of $Z^{\omega}$, $t(Y)=\omega$.

\end{proof}

 A.V. Archangel'skii and V.V. Tkachuk posed the following question (Question 5.16 in \cite{Arh98} and  Problem 4.8.5 in \cite{Tk}):

{\it Let $X$ and $Y$ be Grothendieck spaces. Is then the free topological sum of $X$ and $Y$ a Grothendieck space?}

\medskip

 Note, that the free topological sum of any family of weakly Grothendieck spaces is a weakly Grothendieck space.

 \medskip

The following theorem answers the question in the class of almost discrete spaces.

 \begin{theorem} Let $X=\bigoplus \{X_i: i\in \mathbb{N}\}$ be a free topological sum of almost discrete Grothendieck spaces $X_i$. Then $X$ is a Grothendieck space.
  \end{theorem}


\begin{proof} Note that $X_i$ is a biijective continuous image of one-point
Lindel\"{o}fication $Y_i=L(\tau)$ of a discrete space of size $\tau_i=|X_i|$ for every $i$.
Then $Y=\bigoplus \{Y_i: i\in \mathbb{N}\}$ is a Lindel\"{o}f $P$-space. Since $C_p(X)\subset C_p(Y)$, every compact subset of $C_p(X)$ is Fr\'{e}chet-Urysohn.
By Theorem \ref{th33}, $X$ is Grothendieck.
\end{proof}






\medskip

{\bf Acknowledgements.} I would like to thank Evgenii Reznichenko for several valuable comments.


\bibliographystyle{model1a-num-names}
\bibliography{<your-bib-database>}

\begin{thebibliography}{10}

\bibitem{Os}
Al'perin M., Osipov A.V., Generalization of the Grothendieck's theorem, Topology and its Applications, 338, 2023, 108648.



\bibitem{Arch1}
A.V. Arkhangel'skii, On some topological spaces that arise in functional analysis, Russian Math. Surveys 31:5 (1976) 14--30. (Translated from the Russian.)





\bibitem{Arch79}
A.V. Arkhangel'skii, The spectrum of frequencies of a topological space and the product operation, Tr. Mosk. Mat. Obs., 40, MSU, M., 1979, 171--206 (in Russian).

\bibitem{Arh98}
A.V. Arkhangel'skii, Embeddings in $C_p$-spaces, Topology Appl., 85 (1998), 9--33.

\bibitem{arch}
A.V. Archangel'skii, Topological function spaces. Math. its Appl.,
vol.~78, Dordrecht: Kluwer, 1992, 205~p. ISBN: 0-7923-1531-6\,.
Original Russian text published in {Arkhangel'skii A.V.} {\it
Topologicheskie prostranstva funktsii}, Moscow: MGU Publ., 1989,
222~p.


\bibitem{Arh33}
A.V. Arkhangel'skii, Function spaces in the topology of pointwise convergence, and compact sets, Uspekhi Mat. Nauk, 39:5(239) (1984), 11--50; Russian Math. Surveys, 39:5 (1984), 9--56.




\bibitem{ArchPon}
A.V. Arkhangel'skii, V.I. Ponomarev, Dyadic bicompacta, Dokl. Akad. Nauk SSSR, 182:5 (1968), 993--996 (in Russian).





\bibitem{AsVel}
M.O. Asanov, N.V. Velichko, Compact sets in $C_p(X)$,
Commentationes Mathematicae Universitatis Carolinae, 22:2 (1981), 255--266.





\bibitem{Chr}
J.P.R. Christensen, Joint continuity of separately continuous functions, Proc. Amer. Math. Soc. 82:3 (1981) 455--461.

\bibitem{Eng}
R. Engelking, General Topology, Revised and completed
edition. Heldermann Verlag Berlin, 1989.



\bibitem{Tk}
V.V. Tkachuk, A $C_p$-theorem Problem Book. Compactness in Function Spaces, Springer, (2015).




\bibitem{Tod}
S. Todor\v{c}evi\'{c}, Some applications of $S$ and $L$ combinatorics. The work of Mary Ellen Rudin
(Madison, WI, 1991), 130--167, Ann. New York Acad. Sci., 705, New York Acad. Sci., New
York, 1993.





\bibitem{GN}
J. Gerlits, Z. Nagy, Products of convergence properties. Commentationes Mathematicae Universitatis Carolinae, {\bf 23}:4 (1982), 747--756.


\bibitem{Grot}
A. Grothendieck, Crit\'{e}res de compacticit\'{e} dans les espaces
fonctionnels g\'{e}nereaux, Amer. J. Math., 74, 1952, 168--186.




\bibitem{LM}
A. G. Leiderman, V. I. Malykhin, Nonpreservation of final compactness for the multiplication of spaces of type $C_p(X)$, Sibirsk. Mat. Zh., 29:1 (1988), 84--93; Siberian Math. J., 29:1 (1988), 65--72.




\bibitem{Mal}
V. I. Malykhin, Spaces of continuous functions in elementary generic extensions, Mat. Zametki, 41:4 (1987), 535--542; Math. Notes, 41:4 (1987), 301--304.

\bibitem{Mal2}
V. I. Malykhin, On tightness and Suslin number in $exp X$ and in a product of spaces, Dokl. Akad. Nauk SSSR, 203:5 (1972), 1001--1003.






\bibitem{PS}
D. Preiss, P. Simon, A weakly pseudocompact subspace of Banach spaces is weakly compact, Commentationes Mathematicae Universitatis Carolinae, 15:4 (1974) 603--609.

\bibitem{Pt}
V. Pt\'{a}k, On a theorem of W.F. Eberlein, Studia Math. 14:2 (1954) 276--284.








\end{thebibliography}

\end{document}